\documentclass[12pt]{article}
\usepackage{amsmath,amsthm,amsfonts,latexsym,amsopn,verbatim,amscd,amssymb}

\theoremstyle{plain}
\newtheorem{Thm}{Theorem}[section]
\newtheorem{Lem}[Thm]{Lemma}

\theoremstyle{definition}
\newtheorem{Rem}[Thm]{Remark}

\DeclareMathOperator{\ord}{ord}
\DeclareMathOperator{\Aut}{Aut}

\DeclareMathOperator{\Cl}{C\ell}

\begin{document} 

\title{\textbf{A characterization of certain finite groups of odd order}}  
      
\author{Ashish Kumar Das \ and \ Rajat Kanti Nath}
 
\date{}
\maketitle
\begin{center}\small{\it Department of Mathematics,   North-Eastern Hill University,\\ Permanent Campus, Shillong-793022, Meghalaya, India.\\ Emails:\,  akdasnehu@gmail.com, rajatkantinath@yahoo.com}
\end{center}
\medskip

\begin{abstract} 
The commutativity degree of a finite group  is the probability that two randomly chosen group elements   commute. The main object of this paper is to obtain a characterization for all finite groups of odd order with commutativity degree greater than or equal to $\frac{11}{75}$. 
\end{abstract}

\medskip

 

\medskip
\section{Introduction} \label{S:intro}

Throughout this paper $G$ denotes a finite group with commutator subgroup $G'$ and center $Z(G)$. Recall that
the commutativity degree of $G$  is defined as the ratio
\begin{equation}\label{eqprg}
\Pr(G) =\frac{|\{(x,y) \in G \times G : [x,y] = 1 \}|}{|G \times G|},
\end{equation} 
where $[x,y]= xyx^{-1}y^{-1}$ is the commutator of $x,y \in G$.  Clearly, $\Pr(G) = 1$ if and only if $G$ is abelian.  Gustafson \cite{wG73} has shown that if $G$ is non-abelian, then  $\Pr(G) \leq \frac{5}{8}$. Rusin \cite{dR79}, has computed the values of $\Pr(G)$ when $G' \subseteq Z(G)$, and also  when $G' \cap Z(G) = \{1\}$. Moreover, he has       characterized all $G$ with $\Pr(G) > \frac{11}{32}$. Barry et al. \cite{bM06} have shown that $G$ is supersolvable if $\Pr(G) > \frac{11}{75}$ and $|G|$ is odd.

In this paper,  we   determine the value of $\Pr(G)$ and the size of $\frac{G}{Z(G)}$ when $|G'| = p^2$ and  $|G' \cap Z(G)| = p$, where $p$ is a prime such that $\gcd (p-1,|G|) = 1$. Then, using this and few other supplementary results, we characterize all $G$ of odd order  with $\Pr(G) \geq \frac{11}{75}$.

\section{Some auxiliary results}
   It is well-known that $|\Cl_G (x)|$ divides $|G|$ for all $x \in G$. In fact, we have
\begin{equation}\label{eqcgx}
|G:C_G (x)| \, = \, |\Cl_G (x)| \, = \, |[G,x]|
\end{equation}
for all $x \in G$, where $\Cl_G (x)$ and $C_G (x)$ are   the conjugacy class and the centralizer of $x$ in $G$, and $[G,x] = \{[y, x] \, : \,  y \in G \}$.  It may be mentioned here that $[G,x]$ is not necessarily a subgroup of $G$. 

 Let $p$ be the smallest prime divisor of $|G|$. Then, from \eqref{eqcgx}, we get
\begin{equation}\label{eqclx}
p \,\leq\, |\Cl_G (x)| \,\leq\, |G'|\; \Longleftrightarrow\; x \in G - Z(G).
\end{equation}
Also, if  $|\Cl_G (x)| = p$, for some $x \in G$, then $G/C_G (x)$ forms an abelian group, and so,   $G'  \subseteq  C_G (x)$, which means that  $x \in C_G (G')$,   the centralizer of $G'$ in $G$. Thus, for all $x \in G - C_G (G')$, we have  
\begin{equation}\label{eqpclx}
p \,< \, |\Cl_G (x)|.
\end{equation}
 
 Note that $C_G (G') = \mathop{\cap}\limits_{x \in G'} C_G (x) \,   \trianglelefteq G$. So, by \eqref{eqcgx}, $|\Cl_G (x)|$ divides $|\frac{G}{C_G (G')}|$ for all $x \in G'$.  Moreover, $\frac{G}{C_G (G')}$ is abelian if and only if $G'$ is abelian, and in either case  $[C_G (G'), x]$ is a subgroup of $G'$ for all $x \in G$, a fact that can be easily established using the commutator identities 
\begin{equation}\label{comid}
[xy,z] = x[y,z]x^{-1} [x,z],  \textrm{  and  } [x,yz] = [x,y]y[x,z]y^{-1} 
\end{equation}
where $x,y,z \in G$.  Also, using the Jacobi identity  \cite[page 93]{jR84}, we  have
\begin{equation}\label{eqnilcg}
(C_G (G'))' \subseteq Z(G) \subseteq Z(C_G (G')).
\end{equation}
 
\begin{Lem}\label{lemcent} 
Let $p$ be a prime.  Let   $\Aut (G')$ denote the automorphism group of $G'$ and $C_p$  the cyclic group of order $p$.
\begin{enumerate}
\item  $\frac{G}{C_G (G')}$ can be embedded in $\Aut (G')$, and so, $|\Cl_G (x)|$ divides $|\Aut (G')|$ for all $x \in G'$. \label{lca}
\item  If  $\frac{G}{C_G (G')} \cong C_p$, then $\langle  x ,C_G (G') \rangle  = G $, 
$[G, x] = [C_G (G'), x]$ and $|\Cl_G (x)|$ divides $|G'|$ for all $ x  \in G - C_G (G')$. \label{lcb}
\end{enumerate}  
\end{Lem}

\begin{proof}
Part \ref{lca} follows from \cite[theorem 7.1(i), page 130]{jR84},  and part \ref{lcb} from \eqref{eqcgx} and  \eqref{comid},   noting  that  $[C_G (G'), x]$ is a subgroup of $G'$.
\end{proof}

It is easy to see, from \eqref{eqprg}, that if $H$ and $K$ are two finite groups then  
\begin{equation}\label{eqprod}
\Pr(H\times K) = \Pr(H)\Pr(K).
\end{equation}
In view of this, the following result, which generalizes \cite[lemma 3.8]{bM06}, simplifies our task considerably.  
\begin{Lem}\label{nilcla3}
Let $p$ be a prime such that $\gcd (p-1,|G|) = 1$. If   $|G'|= p^2$ and $|G' \cap Z(G)| = p$,   then $G$ is nilpotent of class $3$; in particular,  $G \cong P \times A$, where $A$ is an abelian group and $P$ is a $p$-group such that $|P'|= p^2$ and $|P' \cap Z(P)| = p$.
\end{Lem}
\begin{proof}
We have
\[
\left|\left(\frac{G}{Z(G)}\right)'\right| = \left|\frac{G'Z(G)}{Z(G)}\right| = \left|\frac{G'}{G' \cap Z(G)}\right| = p.
\]
Therefore, it follows from Lemma \ref{lemcent}\ref{lca} that $G/Z(G)$ is nilpotent of class $2$,  and so, $G$ is nilpotent of class $3$; in particular, $G$ is the direct product of its Sylow subgroups. Hence, the lemma follows.
\end{proof}

Given $H \subseteq G$, consider the set $H^* = \{ x \in G : [G,x] \subseteq H \}$. Then,    we have

\begin{Lem}\label{hstar} 
{\rm\cite[section I]{dR79}}  If $H$, $H_1$, $H_2 \trianglelefteq  G$, then
\begin{enumerate}
\item $(G' \cap H)^* = H^*$, $\{1\}^* = Z(G)$, and $(G')^* = G$,\label{hsa}
\item $(H_1 \cap H_2)^* = H_1^* \cap H_2^*$, and $H_1^*   H_2^* \subseteq (H_1  H_2)^*$,  \label{hsb}
\item  $H_1 \subseteq H_2 \; \Longrightarrow \; H_1^* \subseteq H_2^*$, \label{hsc}
\item $H \trianglelefteq  H^* \trianglelefteq  G$, and $Z(G/H) = H^*/H$,\label{hsd}
\item $G/H^*$  is never a nontrivial cyclic group.\label{hse}
\end{enumerate}
\end{Lem}

The following result is well-known (see \cite{kJ69}), and also can be derived easily from the degree equation   \cite[corollary 2.7]{iM94}.
\begin{Lem}\label{prgp}
Let $p$ be the smallest prime  divisor of $|G|$. Then
\[
 \Pr(G)\leq \frac{1}{|G'|}\left(1+\frac{|G'|-1}{p^2}\right)
\]
with equality if and only if  each non-linear irreducible complex character of $G$ is of degree $p$.
\end{Lem}
  
We conclude this section reformulating a few more known results. 
\begin{Lem}\label{prgnil}
  {\rm\cite[section II]{dR79}}  Let $p$ be a prime and $G$ be a $p$-group with $G' \subseteq Z(G)$. Then
\[
\Pr(G)=\frac{1}{|G'|}\left(1+
       \sum\limits_{\mathop{{K \lneq G'}}\limits_{G'/K \; cyclic}}   \frac{(p-1)|G':K|}{p |G:K^*| } \right),  
\]
where $\frac{G}{K^*} \cong \prod (C_{p^{n_i}} \times C_{p^{n_i}} )$    with  $p \leq p^{n_i} \leq p^{n_1} = p^k = |G':K|$.
\end{Lem}

\begin{Lem}\label{prgrus}
  {\rm\cite[section III]{dR79}}  Let $p$ be a prime. If $|G'|=p$ and $G'\cap Z(G)=\{1\}$, then  $\frac{G}{Z(G)}$ is a non-abelian group of order $pn$ and   
\[
  \Pr(G)\; =\; \frac{n^2+p-1}{pn^2},
\]
where $n > 1$ is a   divisor  of $p-1$.
\end{Lem} 

\begin{Lem}\label{lemZ}
 {\rm\cite[page 303]{bZ98}}  If $G$ is non-abelian, $H \trianglelefteq  G$ and $|G:H|=p$, where $p$ is a prime, then
\[
\Pr(G) = \frac{\Pr(H)}{p^2}+ \frac{p+1}{p |G|^2} \sum\limits_{x \in G-H}  |C_G (x)|.
\]
Moreover, if $H$ is abelian, then  $|\frac{G}{Z(G)}|= p|G'|$, $|C_G (x)|\, =\, |G:G'|$ for all $x \in G-H$,  and  
\[
\Pr(G) = \frac{1}{p^2} + \frac{p^2 -1}{p^2 |G'|}.
\] 
\end{Lem}

   Finally, using GAP \cite{gap08},   we have

\begin{Rem}\label{semidir}
 If $G' \nsubseteq Z(G)$, then $\frac{G}{Z(G)}$ is isomorphic to  $C_7  \rtimes C_3 $, $(C_3 \times C_3)  \rtimes C_3 $,  $ C_{13} \rtimes C_3 $,   $C_{19} \rtimes C_3$,  $C_3\times ( C_7 \rtimes C_3)$ or  $ (C_5 \times C_5)  \rtimes C_3$ according as $|\frac{G}{Z(G)}|$ is $21$,  $27$,  $39$,  $57$,  $63$ or $75$, where   $\rtimes$ stands for  semidirect product \cite[page 137]{jR84}.
\end{Rem}

 \section{Groups with $|G'| = p^2$ and $|G' \cap Z(G)| = p$}
  In this section we derive a formula for $\Pr(G)$ and determine the size of $\frac{G}{Z(G)}$ when $|G'| = p^2$ and  $|G' \cap Z(G)| = p$, where $p$ is a prime such that $\gcd (p-1,|G|) = 1$. 
\begin{Lem}\label{prcg}
 Let $p$ be a prime. If $|G' \cap Z(G)| = p$ and  $C_G (G')$ is non-abelian, then there is a positive integer $s$ such that $\frac{C_G (G')}{ Z(C_G (G'))} \cong (C_p \times C_p)^s$ and
 \[  
 \Pr(C_G (G')) = \frac{1}{p}\left(1 + \frac{p-1}{p^{2s}}\right).  
 \]

\end{Lem}
\begin{proof}
In view of \eqref{eqnilcg}, the result follows from Lemma \ref{prgnil}.
\end{proof}

\begin{Lem}\label{notnil}
Let $p$ be a prime and $G' \nsubseteq Z(G)$ such that  one of the following conditions holds:
\begin{enumerate}
\item $G'  \cong  C_{p^2}$  \ and   \ $\gcd(p-1, |G|)=1$,\label{nna}  
\item $G'  \cong  C_p \times C_p$  \  and  \  $\gcd(p^2 -1, |G|)=1$.\label{nnb}
\end{enumerate}
Then $|\frac{G}{C_G (G')}| = |G' \cap Z(G)|  = |\Cl_G (x)| = p$  for all  $x \in G'-Z(G)$.
\end{Lem}
\begin{proof}
Since  $|\Aut (C_{p^2})| = p(p-1)$, $|\Aut (C_p \times C_p)|= p(p+1)(p-1)^2$ and $G' - Z(G)$ is a union of conjugacy classes of $G$, the result follows from Lemma \ref{lemcent}\ref{lca}.  
\end{proof}
 
\begin{Lem}\label{notnil2}
 If $G' \nsubseteq Z(G)$ and $|\frac{G}{C_G (G')}| = |G' \cap Z(G)| =p$, where $p$ is the smallest prime divisor of $|G|$,  then  $Z(G)^* \subsetneq C_G (G')$  and $\frac{C_G (G')}{Z(G)^*}$   can be embedded in  $\frac{G'}{G' \cap Z(G)}$.
\end{Lem}   
\begin{proof}   By Lemma \ref{hstar}\ref{hse},     $Z(G)^* \neq C_G (G')$. Let $z \in Z(G)^*$.
 Using \eqref{eqcgx} and the definition of $Z(G)^*$, 
  we have $|\Cl_G (z)| \leq |G' \cap Z(G)| = p$. So, by \eqref{eqpclx},   $z \in C_G (G')$. Thus,
   $Z(G)^* \subsetneq C_G (G')$.  Now, let $ x  \in G - C_G (G')$ and note that  $G'$ is
    abelian. Using \eqref{comid}, it is easy to see  that the mapping $f \colon C_G (G') \to G'Z(G)/Z(G)$,
     defined by $f(z) = [z, x ]Z(G)$, is a   homomorphism.  Also, using the definition of $Z(G)^*$ together with  \eqref{eqnilcg} and Lemma \ref{lemcent}\ref{lcb}, we have $\ker f = Z(G)^*$. Thus, it follows that $C_G (G')/ Z(G)^*$ is isomorphic to a subgroup of $\frac{G'Z(G)}{Z(G)} $.  
\end{proof}

\begin{Lem}\label{pgroup}
Let $p$ be a prime. If $G$ is a $p$-group such that $|G' \cap Z(G)| = p$ and $|G'|= p^2$, 
then $Z(G)^* \cap Z(C_G (G')) = G' Z(G)$.
\end{Lem}
\begin{proof}
 By Lemma \ref{notnil}, $\frac{G}{C_G (G')}$ is abelian. Also, by Lemma \ref{nilcla3}, $G$ is nilpotent of class $3$.  So, it follows that $G' Z(G) \subseteq Z(G)^* \cap Z(C_G (G'))$.  For the other direction, choose  $ x  \in G - C_G (G')$ and $y \in G'-Z(G)$ arbitrarily.     Then, by Lemma \ref{lemcent}\ref{lcb},   $C_G ( x ) \cap Z(C_G (G')) = Z(G)$. It follows that  $[y^{-1}, x ] \neq 1$, since $y \notin Z(G)$.  Now, let $z \in Z(G)^* \cap Z(C_G (G'))$. Since   $y,  z   \in   Z(G)^*$, we have   $[z, x ], [y^{-1}, x ] \in G' \cap Z(G)$, a cyclic group of order $p$.  So, there   exists a positive integer $i$ with $1 \leq i \leq p-1$ such that $[z, x ] = [y^{-1}, x ]^i = [y^{-i}, x ]$, or equivalently, $y^i z \in C_G ( x )$. Since  $x \in G - C_G (G')$ is chosen arbitrarily, $y   \in  G' \subseteq Z(C_G (G'))$ and $z \in Z(C_G (G'))$, it follows that  $z \in G' Z(G)$.
 This completes the proof.
\end{proof}
 
We now prove the main result of this section.
\begin{Thm}\label{main}
Let $p$ be a prime such that $\gcd (p-1,|G|) = 1$. If $|G'|= p^2$ and $|G' \cap Z(G)| = p$,  then
\begin{enumerate}
\item $\Pr(G) = 
\begin{cases}
  \frac{2 p^2 - 1}{p^4}  &\text{  if $C_G (G')$ is abelian}\\
\frac{1}{p^4} \left( \frac{p-1}{p^{2s-1}} + p^2 +p -1 \right)     &\text{  otherwise,}
\end{cases}$ \smallskip  \label{mna}
\item $ |\frac{G}{Z(G)} | =
\begin{cases}
 p^3  &\text{  if $C_G (G')$ is abelian}\\
 p^{2s+2} \text{ or } p^{2s+3}   &\text{  otherwise,}
\end{cases}$\smallskip \label{mnb}
\end{enumerate}
where   $p^{2s }=|C_G (G') : Z(C_G (G'))|$.  Moreover, 
\[ \textstyle{
|\frac{G}{G' \cap Z(G)}  : Z(\frac{G}{G' \cap Z(G)})|= |\frac{G}{Z(G)}:Z(\frac{G}{Z(G)})|=p^2.}
\]
\end{Thm}

\begin{proof}
In view of   Lemma \ref{nilcla3}, we can  assume that $G$ is a $p$-group. Therefore, by Lemma \ref{notnil}, we have $|G:C_G (G')| = p$.

\ref{mna}  If $x \in G - C_G (G')$, then it  follows from \eqref{eqpclx} and Lemma \ref{lemcent}\ref{lcb} that  $|\Cl_G (x)| = p^2$, and so,  $|C_G (x)| = |G|/p^2$. Now, the result can be easily deduced from Lemma \ref{lemZ} and Lemma \ref{prcg}.

\ref{mnb} By Lemma \ref{notnil2},  we have  $|C_G (G'): Z(G)^*|=p$. Therefore, using the second isomorphism theorem   \cite[page 25]{jR84}    and Lemma \ref{pgroup}, we have
\[
|Z(C_G (G')):G'Z(G)| = 
\begin{cases}
p &\text{  if $C_G (G')$ is abelian}\\
 1 \text{ or } p   &\text{  otherwise}.
\end{cases}
\]
Hence, using the first part of Lemma \ref{prcg}, the result follows from the normal series
$
Z(G) \subseteq G'Z(G) \subseteq Z(C_G (G')) \subseteq C_G (G') \subseteq G$.

 The final statement follows from Lemma \ref{notnil2} and Lemma \ref{hstar}\ref{hsa}\ref{hsd}.  
\end{proof}

\section{Groups of odd order with $\Pr(G) \geq \frac{11}{75}$}
In this section we accomplish our main objective.
\begin{Lem}\label{lem15-21}
Let $|G|$ be odd.
\begin{enumerate}
\item  If $G' \cong C_{15}$, then $G' \subseteq Z(G)$.\label{1521a}
\item  If $G' \cong  C_{21}$ and $G' \nsubseteq Z(G)$, then 
 $|\frac{G}{C_G (G')}| = |G' \cap Z(G)| = 3$ and 
 exactly one of the following  conditions holds: \label{1521b}
 
 \begin{enumerate}
\item \  $|\Cl_G (x)|= 21$  for all $x \in G - C_G (G')$. \label{1521bi}
\item \ There exists a subset $X$ of \  $G - C_G (G')$ such that 
\begin{eqnarray*}
|X| &\; = \;& 2|Z(C_G (G'))|   \quad \textrm{and  }\\
 |\Cl_G (x)| &\; = \;& 
\begin{cases}
7 &\text{  if   } x \in X  \\
 21   &\text{  if   } x \in G - (C_G (G') \cup X).
\end{cases}
\end{eqnarray*} \label{1521bii}
\end{enumerate}
\end{enumerate}
\end{Lem}
\begin{proof}
Note that   $|\Aut (C_{15})| = 8$ and $|\Aut (C_{21})| = 12$. Therefore, \ref{1521a} and the first part of \ref{1521b} follow from Lemma \ref{lemcent}\ref{lca}.

For the second part of \ref{1521b}, we begin with an observation, which follows from   \eqref{eqpclx} and Lemma \ref{lemcent}\ref{lcb}, that $|\Cl_G (x)|= 7$ or $21$ for all $x \in G - C_G (G')$. 

Assume that the condition \ref{1521bi} fails to hold. Then,  $|\Cl_G (x_0)|= 7$ for some $x_0 \in G - C_G (G')$.  Since $|\frac{G}{C_G (G')}| = 3$, it is easy to see that
\begin{equation}\label{eqcoset}
G - C_G (G') = x_0 C_G (G') \, \sqcup \, x_0 ^{-1}C_G (G').
\end{equation}
Let us write $X\, = \, x_0 Z(C_G (G')) \, \sqcup \, x_0 ^ {-1} Z(C_G (G'))$. Clearly, $X \subseteq G - C_G (G')$ and $|X|    =   2|Z(C_G (G'))|$.  Since $|\Cl_G (x_0)|  = |\Cl_G (x_0 ^{-1})|$, it follows from \eqref{eqcgx}, \eqref{comid} and  Lemma \ref{lemcent}\ref{lcb} that   $|\Cl_G (x)| = 7$ for all $x \in X$.    On the other hand, consider an element $x \in G - (C_G (G') \cup X)$. Using \eqref{eqcoset},  we have  $x=x_0 w$ or $x_0 ^{-1} w$ for some $w \in C_G (G') - Z(C_G (G'))$. Choose $w_1 \in C_G (G') - Z(C_G (G'))$  such that $[w_1 , w] \neq 1$.    Then, by \eqref{comid} and \eqref{eqnilcg}, we have  
\[
[w_1, x] = 
\begin{cases}
[w_1, x_0][w_1, w]   & \text { if } x = x_0 w  \\

[w_1, x_0 ^ {-1}][w_1, w]   & \text { if } x = x_0 ^ {-1} w.
\end{cases}      
\]
Note that $\ord ([w_1, w]) = 3 $  and $\ord([w_1, x_0]) =\ord([w_1, x_0 ^ {-1} ]) = 1$ or $7$, and so,  it follows that  $\ord([w_1, x]) = 3$ or $21$. Thus, using \eqref{eqcgx}, we have $|\Cl_G (x)| =  21 $.   
\end{proof}

 \begin{Lem}\label{lem25}
 If $|G| \equiv  3 \pmod 6$, $G' \cong  C_5 \times C_5$ and $|G' \cap Z(G)| = 1$, then 
 \begin{enumerate}
\item  $|\Cl_G (x)| = 3$ for all $x \in G' -Z(G)$,\label{l25a}
\item $|\frac{G}{C_G (G')}|=3$.\label{l25b}
\end{enumerate}
\end{Lem}
\begin{proof}
 Note that no non-identity element of $G$ is conjugate to its inverse, and also that $G'-Z(G)$ is a union of  conjugacy classes of $G$. 
 
 \ref{l25a} Let $x \in G' -Z(G)$. Since $|G' -Z(G)|=24$, we have $|\Cl_G (x)| \leq 12$. This, in view of Lemma \ref{lemcent}\ref{lca}, implies that $|\Cl_G (x)| = 3$ or $5$.  Since $x^5 =1$,   the classes $\Cl_G (x)$, $\Cl_G (x^2)$, $\Cl_G (x^3)$ and  $\Cl_G (x^4)$ are   distinct but have the same size. Now,   looking at the   partitions of 24 with 3 and 5 as summands, we get  the result.
 
 \ref{l25b} Suppose that there exist   $x,y \in G' -Z(G)$ such that $C_G (x) \neq C_G (y)$. By part \ref{l25a}, we have $|G:C_G (x)|=|G:C_G (y)|=3$. Therefore,
 \[
 \left|\frac{C_G (x)}{C_G (x) \cap C_G (y)}\right| = \left|\frac{C_G (x)C_G (y)}{C_G (y)}\right| = \left|\frac{G}{C_G (y)}\right| =3.
 \]
\noindent This, considering the series $C_G (G') \subseteq C_G (x) \cap C_G (y) \subseteq  C_G (x) \subseteq  G$, implies that $9$ divides $|\frac{G}{C_G (G')}|$. But, in view of Lemma \ref{lemcent}\ref{lca},   $|\frac{G}{C_G (G')}|=$   $3$, $5$, or $15$. Hence, we must have $C_G (x) = C_G (y)$  for all $x, y \in G' -Z(G)$, and so, the result follows.
\end{proof}

We are now in a position to characterize all finite groups of odd order with commutativity degree at least $\frac{11}{75}$.  
 
\begin{Thm}\label{char}
If $|G|$ is odd and $\Pr(G) \geq \frac{11}{75} $, then the possible values of  $\Pr(G)$ and the corresponding structures of $G'$, $G'\cap Z(G)$ and $G/Z(G)$ are  given as follows:
 
\tabcolsep=7pt
\begin{tabular}{cccc}
\hline   $\Pr(G)$ & $G'$  & $G'\cap Z(G)$  & $G/Z(G)$ \\ 
\hline $1$ & $\{1 \}$ & $\{1 \}$ & $ \{1 \}$\\  
\hline    $(1+ {2}/{3^{2s}})/3$    &  $C_3$ & $ C_3 $ & $ (C_3 \times C_3)^s, \; s \geq 1 $\\   
\hline     $(1+ {4}/{5^{2s}})/5 $    & $C_5$  & $  C_5$ & $ (C_5 \times C_5)^s, \; s \geq 1  $\\    
\hline      ${5}/{21} $	 & $C_7$       & $ \{1 \} $ & $ C_7  \rtimes C_3 $\\    
\hline      ${55}/{343} $	 & $C_7$       & $ C_7 $ & $ C_7 \times C_7$\\   
\hline       ${17}/{81} $	 & $C_9$ or $C_3 \times C_3$     & $ C_3  $ &    $(C_3 \times C_3)  \rtimes C_3 $\\   \\
&$C_3 \times C_3 $ & $ C_3 \times C_3$ & $C_3 ^3$\\   
\hline     ${121}/{729} $	 & $C_3 \times C_3$       & $ C_3 \times C_3 $ &    $ C_3 ^4 $\\   
\hline      ${7}/{39} $ 	 & $C_{13}$       &$ \{1 \} $  &   $ C_{13} \rtimes C_3 $ \\   
\hline      ${3}/{19} $	 &  $C_{19}$      & $ \{1 \} $ &   $C_{19} \rtimes C_3$ \\   
\hline      ${29}/{189} $ 	 & $C_{21}$       & $ C_3 $ &   $C_3\times ( C_7 \rtimes C_3)  $ \\   
\hline      ${11}/{75} $	 & $C_5 \times C_5$       & $ \{1 \} $ &    $ (C_5 \times C_5)  \rtimes C_3$ \\   
\hline
\end{tabular}
\end{Thm} 

\begin{proof}
It is enough to assume that $G$ is non-abelian.  By Lemma \ref{prgp}, we have $|G'| \leq 25$. Also, by \cite[lemma 7]{dM93},  $G'$  is not isomorphic to the unique non-abelian group of order $21$. Therefore, it follows that   
 $G'$ is isomorphic to   $C_3 \times C_3$,  $C_9$,  $C_{15}$,  $C_{21}$,  $C_{5} \times C_{5}$,  $C_{25}$   \textrm{ or } $C_p$,
 where $p$ is an odd prime with $p \leq 23$.

\noindent \textit{Case} $1$.  \quad $G' \subseteq Z(G)$.

If $G'  \cong  C_p$,   then, by Lemma \ref{prgnil}, there exists a  positive integer $s$ such that $\frac{G}{Z(G)} \cong (C_p \times C_p)^s$ and
\[
 \Pr(G) = \frac{1}{p}\left(1 + \frac{p-1}{p^{2s}}\right).
 \]
It follows that $s$ can have infinitely many values for $p =3$ and $5$, whereas $s=1$ is the only possibility if $p=7$. For the rest of the values of $p$, we have $\Pr(G) < \frac{11}{75}$.

If $G'  \cong C_3 \times C_3$, then, by Lemma \ref{prgnil},  we have   
\[
\Pr(G) = \frac{1}{9}\left[1 + \frac{2}{3^{2m_1}} + \frac{2}{3^{2m_2}}+ \frac{2}{3^{2m_3}}+ \frac{2}{3^{2m_4}}\right],
\]
with $3^{2m_i} = |\frac{G}{K_i ^*}|$,  $1\leq i \leq 4$, where $K_1$, $K_2$, $K_3$ and $K_4$ are the proper subgroups of $G'$.  A close study of the subgroups $K_1 ^*$, $K_2 ^*$, $K_3 ^*$ and $K_4 ^*$, using Lemma \ref{hstar}, reveals that   $\frac{17}{81}$ (if  $m_1 = m_2 = m_3 = m_4 = 1$)  and  $\frac{121}{729}$ (if   $m_1 = m_2 = 1,m_3 = m_4 = 2$) are the only values of $\Pr(G)$, in the interval [$\frac{11}{75}$, $1$], and on each occasion   $\frac{G}{Z(G)}$ is an elementary abelian $3$-group with  $|\frac{G}{Z(G)}| \leq 81$. In fact, if $\Pr(G) = \frac{121}{729}$, then we have $|\frac{G}{Z(G)}|=81$, and   if $\Pr(G) = \frac{17}{81}$, then equality holds in Lemma \ref{prgp}, and so,  using  \cite[theorem 12.11]{iM94} and the second part of Lemma \ref{lemZ}, we have $|\frac{G}{Z(G)}| = 27$.   

By \eqref{eqprod} and Lemma  \ref{prgnil}, we have $\Pr(G) < \frac{11}{75}$ for the other possibilities of $G'$.

\noindent \textit{Case} $2$.  \quad $G' \cap Z(G) = \{ 1 \}$.

In this case,   by \eqref{eqnilcg}, $C_G(G')$ is an abelian group. Also, in view of Lemma \ref{lemcent}\ref{lca}, Lemma \ref{notnil} and Lemma \ref{lem15-21}, it is not possible to have $|G'|=$ $3$, $5$, $9$, $15$, $17$ or $21$. 

If $|G'|= $ $7$, $11$, $13$, $19$ or $23$,  then   $\Pr(G)$ and  $|G/Z(G)|$ are determined by   Lemma \ref{prgrus}.
More precisely, since there is a unique odd divisor $n>1$ of $p-1$ for each $p \in \{7, 11, 13, 23 \}$, we have $\Pr(G) =\frac{5}{21}$ and $|G/Z(G)| = 21$ if $|G'| = 7 $, and $\Pr(G) =\frac{7}{39}$ and $|G/Z(G)| = 39$ if $|G'| = 13 $, while   $\Pr(G) < \frac{11}{75}$    if $|G'| = 11$ or $23$.    On the other hand, if $p= 19$, then there are two such odd divisors ($n=3$ and $n=9$) of  $p-1$. It can be seen that if $n=3$, then $\Pr(G) =\frac{3}{19}$ and $|G/Z(G)| = 57$, whereas $\Pr(G) < \frac{11}{75}$ if $n=9$.

If $|G'|= 25$, then, in view of Lemma \ref{notnil}, we must have $G' \cong  C_5 \times C_5$  and $|G| \equiv  3 \pmod 6$.   So, using Lemma \ref{lem25}\ref{l25b} and the second part of Lemma \ref{lemZ}, we have $\Pr(G)=\frac{11}{75}$ and   $|\frac{G}{Z(G)}| = 75$.

\noindent \textit{Case} $3$.  \quad $G' \nsubseteq Z(G)$ and $G' \cap Z(G) \neq \{ 1 \}$.

In this case, $|G'|=$ $9$, $21$ or $25$ ( $|G'| \neq 15$, by Lemma \ref{lem15-21}\ref{1521a}.

If $|G'|=9$, then $|G' \cap Z(G)| = 3$. Using Theorem \ref{main}, we see that $\frac{17}{81}$ is  the only value of $\Pr(G)$, in the interval [$\frac{11}{75}$, $1$], and it  occurs when $C_G (G')$ is abelian. Also, in that case, we have $|\frac{G}{Z(G)}|= 27$.

If $|G'|=21$, then, by Lemma \ref{lem15-21}\ref{1521b}, $|G' \cap Z(G)| = 3$.  Using Lemma \ref{lemZ} together with Lemma \ref{prcg} and Lemma \ref{lem15-21}\ref{1521b}, we see  that $\frac{29}{189}$ is  the only value of $\Pr(G)$, in the interval [$\frac{11}{75}$, $1$], and it occurs   when $C_G (G')$ is abelian. Also, in that case $|\frac{G}{Z(G)}|=63$.  

Finally, if $|G'|=25$, then $|G' \cap Z(G)| = 5$. But, using Theorem \ref{main}, we see that $\Pr(G)< \frac{11}{75}$ .

 Thus, in view of Remark \ref{semidir}, the theorem is completely proved. 
\end{proof}

 We conclude our discussion with the following remark.
\begin{Rem}\label{lastrem}
In \cite[section IV]{dR79},  Rusin  classifies all $G$ with $\Pr(G) > \frac{11}{32}$. However, there are a few of points that are worth noting.
\begin{enumerate}
\item In case 2, he   misses out one situation, where we have $\Pr(G) =\frac{5}{14}$,    $G'  \cong  C_7$, $ G' \cap Z(G) = \{ 1 \}$, and $ \frac{G}{Z(G)}  \cong D_7$,  the dihedral group of order $14$.

\item  In case 3, he claims to have shown that if $|G'|= 4$ and $|G' \cap Z(G)| = 2$, then 
$
 \Pr(G)= {1}/{4} \cdot \left(1 +  {1}/{2^{2t}} +  {1}/{2} \cdot  {1}/{2^{2s}}\right) 
$ 
with $2^{2s}=[C_G(G'): Z(C_G(G'))]$,  $2^{2t} = |\frac{G}{G' \cap Z(G)}  : Z(\frac{G}{G' \cap Z(G)})|$  and  $s+1\geq$ $ t\geq 1$. But,  putting $p=2$ in Theorem \ref{main},  we see that $t=1$ is the only possibility.

\item In the summary,   he writes that if  $C_2 \times C_2$ $ \cong  G' \subseteq Z(G)$, then $\frac{G}{Z(G)} \cong  C_2 ^3$ or $C_2 ^4$, and $\Pr(G) = \frac{7}{16}$ or $\frac{25}{64}$.  However, arguing in the same manner as we have done in a similar situation (namely, $C_3 \times C_3 \cong  G' \subseteq Z(G)$) in the proof of Theorem \ref{char}, we see  that $\frac{G}{Z(G)} \cong  C_2 ^3$ or $C_2 ^4$ according as $\Pr(G) = \frac{7}{16}$ or $\frac{25}{64}$.  \  Rusin also writes that if $\Pr(G) = \frac{3}{8}$,  $G' \cong C_6$  and $G' \cap Z(G) \cong C_2$, then $\frac{G}{Z(G)} \cong C_2 \times S_3$  or $T$  , where $T$ is the non-abelian group of order $12$ besides $A_4$ and $C_2 \times S_3$.  But, it is well-known that $T \not\cong \frac{G}{Z(G)}$ for any $G$.\label{lremd}
\end{enumerate}
\end{Rem}


\begin{thebibliography}{6}

\bibitem{bM06}
F. Barry,  D.  MacHale,  and $\mathrm{\acute A}$. N$\mathrm{\acute i}$ Sh$\mathrm{\acute e}$, 2006  Some supersolvability conditions for finite groups.  \textit{Mathematical Proceedings of the Royal Irish  Academy\/} \textbf{106A},   163--177.  
 
\bibitem{bZ98}
Ya. G. Berkovich, and E. M.  Zhmud$^{\prime}$,  1998    \textit{Characters of Finite Groups. Part 1}. Translations of Mathematical Monographs \textbf{172}, American Mathematical Society, Providence, RI.
  
\bibitem{wG73}
W. H. Gustafson, 1973    What is the probability that two group elements commute?  \textit{American Mathematical Monthly\/} \textbf{80}, 1031--1034.

\bibitem{iM94}
I. M. Isaacs,  1994  \textit{Character theory of finite groups}. Dover Publications. Inc. New YorK. 

\bibitem{kJ69}
K. S.  Joseph, 1969  Commutativity in non-abelian groups. Unpublished Ph. D. thesis. University of California. Los Angeles.
 
\bibitem{dM93}
D. Machale, and  P. $\acute{O}$ Murch$\acute{u}$,  1993  Commutator subgroups of groups with small central factor groups. \textit{Mathematical Proceedings of the Royal Irish  Academy\/}  \textbf{93A}, 123--129.
 
\bibitem{jR84}
J. J. Rotman, 1984  \textit{An introduction to the theory of groups\/} (3rd edn). Allyn and Bacon. Inc.

\bibitem{dR79}
D. J. Rusin, 1979    What is the probability that two elements of a finite group commute?  \textit{Pacific Journal  Mathmatics\/} \textbf{82}, 237--247.

\bibitem{gap08}
The GAP~Group, 2008 GAP -- Groups, Algorithms, and Programming. Version 4.4.12.   http://www.gap-system.org  

 
 

\end{thebibliography}
\end{document}